\documentclass{article}
\usepackage{amsmath}
 \usepackage{amsfonts}
\usepackage{amsthm}
\usepackage{bbm}
\usepackage{hyperref}
\usepackage{tikz}
\usepackage{float}
\theoremstyle{plain}

\numberwithin{equation}{section}

\newtheorem{thm}{Theorem}[section]
\newtheorem{remark}[thm]{Remark}
\newtheorem{prop}[thm]{Proposition}

\newtheorem{cor}[thm]{Corollary}
\newtheorem{lem}[thm]{Lemma}

\begin{document}
\noindent

\title{The values of a family of Cauchy transforms}
\author{Kevin F. Clancey}
\date{July 2022}
\maketitle
\begin{abstract} The family of Cauchy transforms \[C_{g}(z,w) =  -\frac{1}{\pi}\int_{\mathbb{C} }
\frac{g(u)}{\overline{u-w} (u-z) }  da(u ),\] where the measurable function $g$ with compact (essential) support satisfies $0 \leq g\leq 1,$ and suitably defined for all complex $z, w,$ is closely connected to the theory of Hilbert space operators with one-dimensional self-commutators.  Based on these connections one can derive the inequality \[\vert 1-\exp C{g}(z,w)\vert\leq 1. \] Here, using elementary methods, a direct proof of this inequality is given. The approach involves a detailed study of the convex family of integrals \[I_{g}= -\frac{1}{\pi}\int_{\mathbb{C} }
\frac{g(u)}{\overline{u+1} (u-1) } da(u),\] where $g$ varies over the set of measurable functions with compact support satisfying $0 \leq g\leq 1.$ These integrals are transformed to a tractable form using a parametriztion of the plane minus the real axis using the family of circles passing though the points $+1,-1.$ The characeristic functions of discs bounded by these circles are unique points in the boundary of the convex set of values of the family of integrals.
\end{abstract}

\section{Introduction}Let $\mathcal{G}_1$ denote the equilvalence classes of compactly supported bounded measurable functions $g$ on the complex plane $\mathbb{C}$ that satisfy $0\leq g\leq1 . $ For $g\in\mathcal{G}_1$  let $E_{g}(z,w)=E(z,w)$ be defined by 
\begin{equation}\begin{gathered}\label{key} E(z,w) = \exp- \frac{1}{\pi}\int_{\mathbb{C} }
\frac{u-w}{u-z }
\frac{g(u)}{\vert u-w\vert^2} da(u )= \\ \exp -\frac{1}{\pi}\int_{\mathbb{C} }
\frac{g(u)}{\overline{u-w} (u-z) }  da(u )\end{gathered}\end{equation}
for $z\neq w,$ with $E(w,w)$ defined to be $0$ if $\frac{1}{\pi}\int_{\mathbb{C}}g(u)\vert u-w\vert^{-2} da(u) =+ \infty$ and equal to \[\exp-\frac{1}{\pi}\int_{\mathbb{C}} \frac{g(u)}{\vert u-w\vert^2} da(u )\] when \begin{equation}\label{finite}\frac{1}{\pi}\int_{\mathbb{C}} \frac{g(u)}{\vert u-w\vert^2} da(u )<\infty .\end{equation}  Here $a$ denotes area measure.

The function $E_g$ has a basic connection with the study of bounded linear operators on a Hilbert space with one-dimensional self-commutator. Here these connections will only be described briefly at the end of Section 2. For a more detailed discussion see \cite{Clancey} and, in particular, the references cited in \cite{Clancey}. For now we note that this operator theory connection implies function theoretic properties of $E_g$ that are not obvious nor easily established independent of this connection. In \cite{Clancey} a direct study of such properties was initiated. In particular, with the above conventions on $\frac{1}{\pi}\int_{\mathbb{C}}g(u)\vert u-w\vert^{-2} da(u)$ a direct proof of the sectional continuity of $E_g(z,w)$ was given.  The goal here is to provide elementary methods to establish 

\begin{equation}\label{ineq} \vert 1-\exp C{g}(z,w)\vert\leq 1  \end{equation}
without connections to operator theory. 

A direct function theoretic investigation of this last inequality involves a detailed description of the values of
\begin{equation}\label{values}C_{g}(z,w) =  -\frac{1}{\pi}\int_{\mathbb{C} }
\frac{g(u)}{\overline{u-w} (u-z)}  da(u), g\in\mathcal{G}_1, z,w\in\mathbb{C} .\end{equation} For w fixed, as a functions of $z,$ the above integrals are the complex Cauchy transforms of the functions $g_{w}(u)=\frac{g(u)}{\overline{u-w}}.$ The conventions on the case when $z=w$  imply that (\ref{ineq}) holds in that case.
For $z\neq w$ the linear change of variables $u(v)=\frac{1}{2}((z-w)v+(z+w)),$ which sends $-1$ to $w$ and $1$ to $z,$  transforms this last integral to the form \begin{equation} -\frac{1}{\pi}\int_\mathbb{C}\frac{g(u(v))}{\overline{v+1}(v-1)}da(v).\end{equation}The potential singularities at $u=z,w$ and $v=-1,1$ in these last two integrands necessitates some care with this assertion. One way to deal with this is to observe that for small $\epsilon >0$ the identity holds off the set of $v=re^{it}$ for $(r,t)\in (1-\epsilon, 1+\epsilon)\times ((-\epsilon,\epsilon)\cup(\pi-\epsilon, \pi+\epsilon)).$ As $\epsilon\rightarrow 0$ the corresponding $u$ and $v$  integrals converge to the integrals over $\mathbb{C}.$

 Thus describing the set of values in (\ref{values}) is equivalent to describing the following convex set of values \begin{equation}\label{reducedvalues}I_{g}=C_{g}(1,-1) =  -\frac{1}{\pi}\int_{\mathbb{C} }
\frac{g(u)}{\overline{u+1} (u-1)}  da(u),\  g\in\mathcal{G}_1 .\end{equation} The path to this description involves an explicit description of the set of boundary values of this convex set that in turn produce a fortuitous chage of variables for the integrals (\ref{reducedvalues}).

Key to understanding the set of values $ I_{g}, g\in\mathcal{G}_{1}$ will be the family $g_\theta = \chi_{D_\theta}, 0<\theta <\pi,$ where $D_\theta$ is the disc bounded by the circle $\Gamma_{\theta}$ of radius $\csc (\theta)$ with center $i\cot (\theta).$ Note that the notation $I_{\theta}$ will be used for $I_{g_{\theta}}.$ These circles pass through the points $-1,1$ and subtend the angle $\theta$ over the chord through these points. Each point $u$ in the complex plane off the real axis is on precisely one circle $\Gamma_{\theta (u)}.$
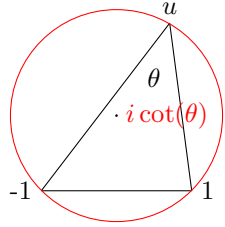
\begin{figure}[H]
\centering
\begin{tikzpicture}
\draw (-1,0) node[left] {-1}- - (1,0) node[right]{1};
\draw (-1,0) -- (.707,2.23);
\draw (1,0) --  (.707,2.23);
\draw[red] (0,1) node[right]{$i\cot(\theta)$} (0,1) circle (1.41);
\draw (0,1) circle (0.01);
\draw (.5,1.5) node {$\theta$};
\draw(.707,2.23) node[above]{$u$};
\end{tikzpicture}
\caption{$\Gamma_{\theta (u)}$}
\end{figure}

Using the Law of Cosines one sees the value of the kernel \begin{equation}\label{ker} k(u)= -\frac{1}{\pi}\frac{1}{\overline{u+1}(u-1)},\: u\neq\pm 1\end{equation} on the circle $\Gamma_{\theta (u)}, \theta\neq\frac{\pi}{2},$ is \begin{equation}\label{rep} k(u)= \frac{1}{\pi}\frac{\cos(\theta)e^{-i\theta}}{1-\vert u\vert ^2} .\end{equation} For $u$ on the unit circle $u\ne\pm 1$\begin{equation}\label{rep2} k(u)=\frac{i}{2\pi}\frac{1}{Im  (u)}.\end{equation}

This family of circles will be used to parametrize the plane minus the real axis. Below in Figure 2  is a picture of this family of circles where the case $\theta =\frac{\pi}{2}$ plays a central role:

\begin{figure}[H]
\centering
\begin{tikzpicture}

\draw (-1,0) node[left] {-1}- - (1,0) node[right]{1};
\draw(0,0) node {0} circle (1);
\draw[red] (0,.58) node{$ic_{\theta}$}circle (1.16);\draw[blue] (0,-.58)circle (1.16);\draw (0,-2)- - (0,2);\draw[red] (0,1) circle (1.41);
\draw[blue](0,-1) circle (1.41);
\draw [blue] (2,-3) node[right]{$\downarrow\pi$} --(2,0)node[right, black]{$\frac{\pi}{2}$}node[left,black]{$\theta$};
\draw [red] (2,0)--(2,3) node[right]{$\uparrow 0$};
\end{tikzpicture}\caption{The family of circles $\Gamma_\theta$ }
\end{figure}
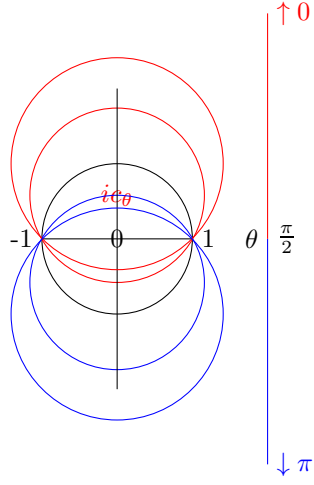

One of the main results here is that the value \begin{equation}\label{keyvalue} I_{\theta}  =  -\frac{1}{\pi}\int_{D_{\theta} }
\frac{1}{\overline{u+1} (u-1)}  da(u)\end{equation} is unique in the set of values (\ref{reducedvalues}). A proof of this uniqueness will be given in the next section. But first, we compute the value $I_{\theta}.$ The identity $I_{\pi-\theta}=\overline{I_{\theta}}$ allows one to restrict to the case $0<\theta\leq\frac{\pi}{2}.$ This computation can be accomplished by using the polar coordinates \[ z(r,t) = re^{it}+i\cot(\theta), (r,t)\in (0,\csc(\theta))\times (-\pi, \pi). \] Some care must be taken with the singularities at $u=\pm 1.$  This can be handled by first computing the integral (\ref{keyvalue}) over discs of radius $R$ centered at $i\cot(\theta),$ where $0<R<\csc(\theta)$ and letting $R$ approach $csc(\theta).$ We proceed formally as follows. The above polar coordinates permit one to write the integral (\ref{keyvalue}) in the form
\begin{equation}I_{\theta}= \int_{0}^{\csc (\theta)}\frac{1}{2\pi i}\int_{\vert z\vert =1}\frac{2r}{(r+\tau z)(\tau-rz)}dzdr, \end{equation} where $\tau = i\cot\theta-1=i\csc(\theta) e^{i\theta}.$ For $0\leq r< \csc\theta$ the Cauchy integral formula gives
\begin{equation}I_{\theta}= \int_{0}^{\csc (\theta)}\frac{2r}{r^2+\tau^{2}}dr. \end{equation} Using the principal valued  logarithm function $Log (r^2+\tau^{2})$  ($-\pi <Arg z\leq\pi$) as an antiderivative for this last integral, one finds \begin{equation}\label{keyvalue2} I_{\theta}  =  -\frac{1}{\pi}\int_{D_{\theta} }
\frac{1}{\overline{u+1} (u-1)}  da(u) = \ln(2\sin\theta)+i(\frac{\pi}{2}-\theta).\end{equation} The identity $I_{\pi-\theta}=\overline{I_{\theta}}$  shows this last formula for $I_{\theta}$ is valid for all $\theta$ in the range $0<\theta <\pi.$ 

 The complex plot of the values $I_{\theta}$ for $0<\theta <\pi$ is the following

\begin{figure}[H]
\centering
\label{keyI}

\begin{tikzpicture}
\draw (-2,0)node {$\Omega_{1}$};

\draw[ domain=-1.57:1.57, smooth, variable=\y, blue] plot ({ln(2*cos(\y r)}, {\y});
\draw (0,-1.57) node[below] {$-i\frac{\pi}{2}$} -- (0,1.57) node[above]{$i\frac{\pi}{2}$};
\draw [color=red] (-4,-1.57) -- (1,-1.57);
\draw [color=red] (-4,1.57) -- (1,1.57);
\draw (0,0)node[right]{0};
\draw (.7,0)node[right]{ln2};

\end{tikzpicture}\caption{Plot of $I_{\theta}$}
\end{figure}
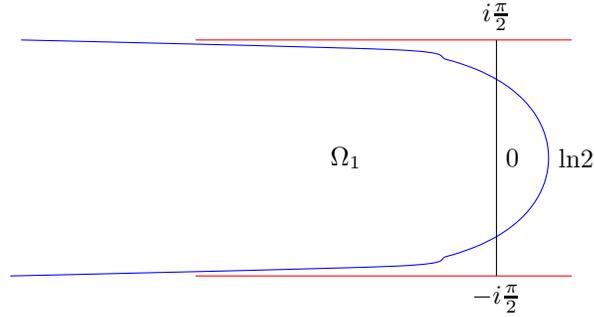 If one parametrizes the circle $\vert z-1\vert =1$ with \[z(\theta) =\  1- e^{-2i\theta},\    0\leq \theta<\pi\] then $Log (z(\theta))= I_\theta, \theta\neq 0.$ That is, the above plot is the image of this circle minus the origin under the principal valued logarithm function. Under this logarithm mapping the inside of the circle is mapped to the open convex region $\Omega_1$ bounded by this plot. Establishing the inequality (\ref{ineq}) is equivalent to showing that for $g\in \mathcal{G}_{1}$ the value $I_{g}$ is in the closure of $\Omega_1.$
\begin{remark} The coarse estimates $\Re (I_g)\leq\ln 2$ and $\vert\Im (I_g)\vert<\frac{\pi}{2}$ that place $I_g$ in the strip $(-\infty, \ln2)\times (-\frac{\pi}{2}, \frac{\pi}{2})$ can be established more directly. See, below.\end{remark}

\section{The set of values $I_g, \ g\in\mathcal{G}_1$} The goals of this section are the following. First to show that if for $g\in \mathcal{G}_1$ and $\theta\in (0,\pi)$ there holds $I_g=I_\theta,$ then $g=g_\theta.$ Second, to show that for $g\neq g_\theta$ the value $I_g$ is in $\Omega_1.$ Once the first goal, i.e. the uniqueness of $I_\theta,$ is established the second follows relatively easily. To establish the first goal a parametrization of the complex plane minus the real axis using the family of circles $\Gamma_\theta, 0<\theta <\pi$ will be used. This parametrization serendipitously turns the integral $I_g$ into a form that provides tractable estimates.

\subsection{A change of variables formula for $I_{g}$}Consider the map \begin{equation}\label{map} u(t,\theta) = \csc(\theta)e^{it} + i\cot(\theta),  (t,\theta)\in (-\pi,\pi ]\times (0,\pi)\end{equation} from the cylinder $\mathcal{K}=(-\pi,\pi]\times (0,\pi)$ to the complex plane. The map $u=u(t,\theta)$ maps the lines $\theta =- t- \frac{\pi}{2},$ $\theta=t+\frac{\pi}{2},$ and $\theta=-t+\frac{3\pi}{2}$ to zero. These lines constitute the zero set $\mathcal{Z}$ of $\sin t +\cos\theta$ on the cylinder. The two open trangular regions on the cylinder bounded by these lines are mapped smoothly one-to-one and onto the upper and lower half-planes. These regions are marked $\mathcal{U}$ and $\mathcal{L}$ in the figure below.

\begin{figure}[H]
\centering
\label{domain}
\begin{tikzpicture}

\draw[dotted] (-5,0)node[left]{$(-\pi, 0)$}--(0,0)node[below]{(0,0)}--(5,0)node[below]{$(\pi,0)$}node[right]{t};
\draw[dotted](-5,0)--(-5,5)node[left]{$(-\pi,\pi)$}node[above]{$\theta$};
\draw[dotted](-5,5)--(5,5);
\draw (5,0)--(5,5);
\draw[color=red][thick] (-5,2.5)--(-2.5,0);
\draw[color=red][thick] (-2.5,0)--(2.5,5);
\draw[color=red][thick] (2.5,5)--(5,2.5);
\draw[color=blue](2.5,2.5)node{$\mathcal{U}$};
\draw[color=blue](-2.5,2.5)node{$\mathcal{L}$};
\draw[color=blue](-4,1)node{$\mathcal{U}$};
\draw[color=blue](4,4)node{$\mathcal{L}$};
\draw(-5,2.5)node[left]{$(-\pi,\frac{\pi}{2})$};
\draw (-2.5,0)node[below]{$(-\frac{\pi}{2},0)$};
\draw (2.5,0)node[below]{$(\frac{\pi}{2},0)$};
\end{tikzpicture}\caption{The domain cylinder $\mathcal{K}$}
\end{figure}
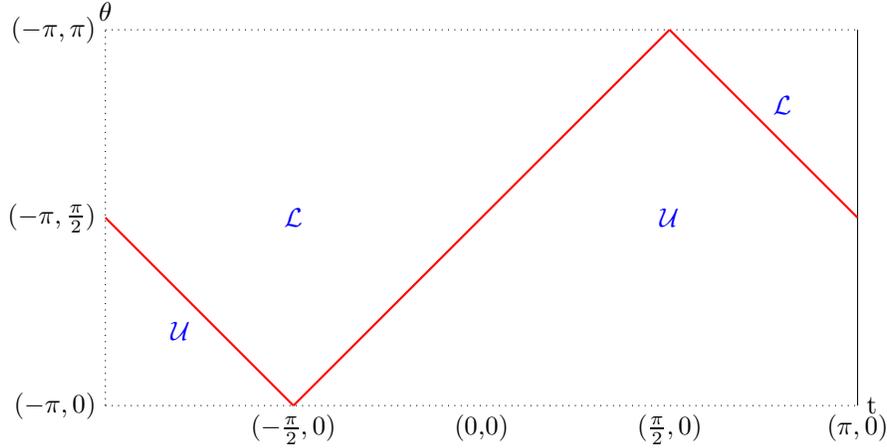

Computation of the Jacobian yields the change of variables formula

\begin{equation}
\label{change}
da(u)=\csc^3 (\theta)\vert\sin t + \cos\theta\vert dtd\theta
\end{equation}
separately on $\mathcal{U}$ and $\mathcal{L}.$ Note that the term $\sin t + \cos\theta$ is positive on $\mathcal{U}$ and negative on $\mathcal{L}.$

Using the formulae (\ref{rep}) and (\ref{rep2}) for the kernel $k$ one finds \begin{equation}\label{parametrized kernel} k(u(t,\theta))=\frac{-1}{2\pi}\frac{\cot\theta - i}{\csc^3(\theta) (\sin t+ \cos\theta)}.\end{equation} this last formula is valid off of the set $\mathcal{Z}.$ The  fortunate occurence of the denominator factor in this last formula matching the change of variables factor in (\ref{change}) alows one to write the integral (\ref{reducedvalues}) in the form
\begin{equation}
\label{iden}I_{g} = \frac{1}{2\pi} \int_{\mathcal{U}}\ (-\cot(\theta) + i) g(u(t,\theta ))dtd\theta + \frac{1}{2\pi} \int_{\mathcal{L}} (\cot(\theta) - i) g(u(t,\theta ))dtd\theta .
\end{equation}
At this point (\ref{iden}) is only claimed to be true for $g$ with compact essential support off the real axis.
We first address the existence of the integral on the right side of this last identity.
\begin{lem} Let $g$ be a bounded measurable function with bounded essential support in the plane. Then the integral \begin{equation} \frac{1}{2\pi} \int_{\mathcal{U}\cup\mathcal{L}} \vert\cot(\theta)\vert  g(u(t,\theta ))dtd\theta\end{equation} exists and is finite.\end{lem}
\begin{proof} It follows from the comparision test for improper integrals that it is sufficient to establish the claim for the case where $g$ is the characteristic function of a disc centered at $0$ of radius $R>1.$ Indeed, using symmetry it is sufficient to establish the result in this case for the pre-image of  of this disc under $u(t,\theta)$ in $\mathcal{U}.$ Let $M=\frac{R^2+1}{2}.$ 
The portion of the preimage of the open disc of radius $R$ in $\mathcal{U}$ is the domain $\mathcal{R}_{+}$ bounded by the lines $t=t_{1}(\theta)=\theta-\frac{\pi}{2},$ and $t=t_{2}(\theta)=\frac{3\pi}{2}-\theta$ for $0<\theta<\pi$ and the curves  \begin{equation}s_{1}(\theta)=\arcsin(\frac{M\sin^2(\theta)-1}{\cos(\theta)}), \ s_{2}(\theta)=\pi-s_{1}(\theta)\end{equation} defined for $0<\theta<\theta_{R}\leq\arccos(\frac{R^2-1}{R^2+1}).$  This value $\theta_{R}$ is the largest positive value of $\theta$ where the circle $\Gamma_{\theta}$ intersects the closed disc of radius $R.$ The notation $\theta_{(R,0)}$ will be used for the value $\arcsin(\frac{2}{M})^{\frac{1}{2}},$ where $s_{1}(\theta) =0.$  Fubini's Theorem can be used to estimate \begin{equation}
\int_{\mathcal{R}_{+}}\vert\cot(\theta)\vert d\theta dt. \end{equation} This estimate will be handled in three portions. For the portion of $\mathcal{R}_+$ in the range $\frac{\pi}{2}<\pi$ the iterated integral has the form \begin{equation} -\int_{\frac{\pi}{2}}^{\pi} \cot(\theta)\left(\int_{t_{1}(\theta)}^{t_{2}(\theta)} dt\right) d\theta = -\int_{\frac{\pi}{2}}^{\pi}\cot(\theta) 2(\pi -\theta) d\theta\end{equation} which is finite. For the portion of $\mathcal{R}_{+}$ in the range  $\theta_{(R,0)} \leq\theta\leq\frac{\pi}{2}$ the function $\cot(\theta)$ is bounded and the corresponding integral is finite. For the portion of the integral in the range $0<\theta<\theta_{(R,0)}$ (the region denoted $\mathcal{S}$ in the schematic figure below) the integral has the form
 \begin{equation} 2\int_{0}^{\theta_{(R,0)}} \cot(\theta)\left(\int_{t_{1}(\theta)}^{s_{1}(\theta)} dt\right) d\theta =2 \int_{0}^{\theta_{(R,0)}}\cot(\theta) (s_{1}(\theta)-t_{1}(\theta)) d\theta,\end{equation} where the factor $2$ arises from the symmetry of $\mathcal{R}_{+}$ in the line $t=\frac{\pi}{2}.$ The proof will be complete once it is shown that $\cot(\theta) (s_{1}(\theta)-t_{1}(\theta)) $ is bounded on $0<\theta<\theta_{(R,0)}.$ This follows from the estimate \begin{equation} \begin{gathered}
s_{1}(\theta) -t_{1}(\theta)=\arcsin(\frac{M\sin^2(\theta)-1}{\cos(\theta)}) -\arcsin(-\cos(\theta)) \leq \\ \frac{1}{(1-\cos^2(\theta))^{\frac{1}{2}}}\frac{(M-1)\sin^{2}(\theta)}{\cos(\theta)}= \\ \frac{(M-1)\sin(\theta)}{\cos(\theta)},\end{gathered}\end{equation} where the first inequality follows from the Mean Value Theorem applied to the $\arcsin$ function on the interval \begin{equation} \left[-\cos(\theta), \frac{M\sin^2(\theta)-1}{\cos(\theta)}\right].
\end{equation}.
\end{proof}
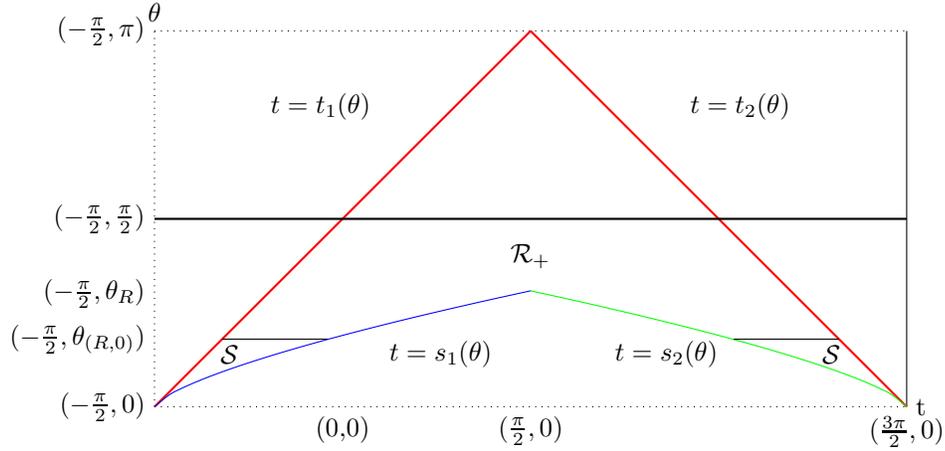
\begin{figure}[H]
\centering
\label{domain1}
\begin{tikzpicture}
\draw[dotted] (-5,0)node[left]{$(-\frac{\pi}{2}, 0)$}--(0,0)node[below]{$(\frac{\pi}{2},0)$}--(5,0)node[below]{$(\frac{3\pi}{2},0)$}node[right]{t};
\draw(-2.5,0)node[below]{(0,0)};
\draw[dotted](-5,0)--(-5,5)node[left]{$(-\frac{\pi}{2},\pi)$}node[above]{$\theta$};
\draw[dotted](-5,5)--(5,5);
\draw (5,0)--(5,5);
\draw[color=red][thick](-5,0)--(0,5);
\draw[color=red][thick](0,5)--(5,0);
\draw[color=black][thick](-5,2.5)--(5,2.5);
\draw[ domain=-5:0, smooth, variable=\x, blue] plot ({\x},{.5*(\x+5)^.7});
\draw[ domain=0:5, smooth, variable=\x, green] plot ({\x},{.5*(-\x+5)^.7});
\draw(-2,4)node[left]{$t=t_{1}(\theta)$};
\draw(2,4)node[right]{$t=t_{2}(\theta)$};
\draw(-2,.7)node[right]{$t=s_{1}(\theta)$};
\draw(1,.7)node[right]{$t=s_{2}(\theta)$};
\draw(-5,2.5)node[left]{$(-\frac{\pi}{2},\frac{\pi}{2})$};
\draw(0,2)node{$\mathcal{R}_{+}$};
\draw(-5,1.5)node[left]{$(-\frac{\pi}{2},\theta_{R})$};
\draw(-5,.9)node[left]{$(-\frac{\pi}{2},\theta_{(R,0)})$};
\draw(-4,.45)node[above]{$\mathcal{S}$};
\draw(4,.45)node[above]{$\mathcal{S}$};
\draw(-4.1,.9)--(-2.7,.9);
\draw(2.7,.9)--(4.1,.9);
\end{tikzpicture}
\caption{Preimage of upper-half of disc of radius $R$}
\end{figure}

It now follows from a limit argument that the identity (\ref{iden}) holds for any bounded measurable function $g$ with compact essential support.
\begin{remark}  The coarse estimates $\Re (I_g)\leq\ln 2$ and $\vert\Im (I_g)\vert<\frac{\pi}{2}$ can be established using (\ref{iden}). For example, the estimate  $\vert\Im (I_g)\vert<\frac{\pi}{2}$ follows easily from the fact that the areas of the triangles $\mathcal{U}$ and $\mathcal{L}$  are $\pi^{2}.$

\end{remark}
\subsection{Uniqueness of $I_{\theta}$}

We now turn to our first goal in this section of establishing that $I_{\theta}$ is uniquely attained when $g=g_{\theta}.$
\begin{prop} Let $g$ be in $\mathcal{G}_{1}.$ If $I_{g} = I_{\phi}$ for some $\phi\in (0,\pi),$ then $g=g_{\phi}.$
\end{prop}
\begin{proof}
It is sufficient to consider the case $0<\phi<\frac{\pi}{2}.$ We decompose the complex plane minus the real axis into 6 open sets $A--F$ as shown in the figure below.
 
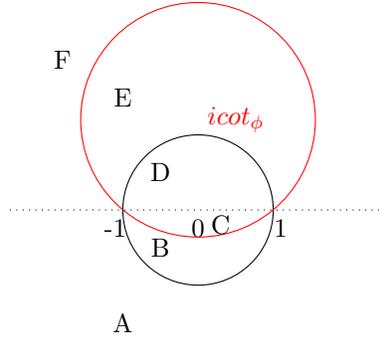
\begin{figure}[H]
\centering
\begin{tikzpicture}

\draw[dotted] (-2.5,0)- - (2.5,0);
\draw(0,0) node[below]{0} circle (1);
\draw[red] (0,1.2) node[right]{$icot_{\phi}$};\draw[red] (0,1.2) circle (1.56);
\draw(-1,-1.5)node{A};
\draw(-.5,-.5)node{B};
\draw(.3,-.2)node{C};
\draw(-.5,.5)node{D};
\draw(-1,1.5)node{E};
\draw(-1.8,2)node{F};
\draw(-1.1,0)node[below]{-1};
\draw(1.1,0)node[below]{1};

\end{tikzpicture}
\end{figure}
The corresponding decomposition $\mathcal{A}--\mathcal{F}$ on the $(t,\theta)$ domain cylinder is shown in the figure below.
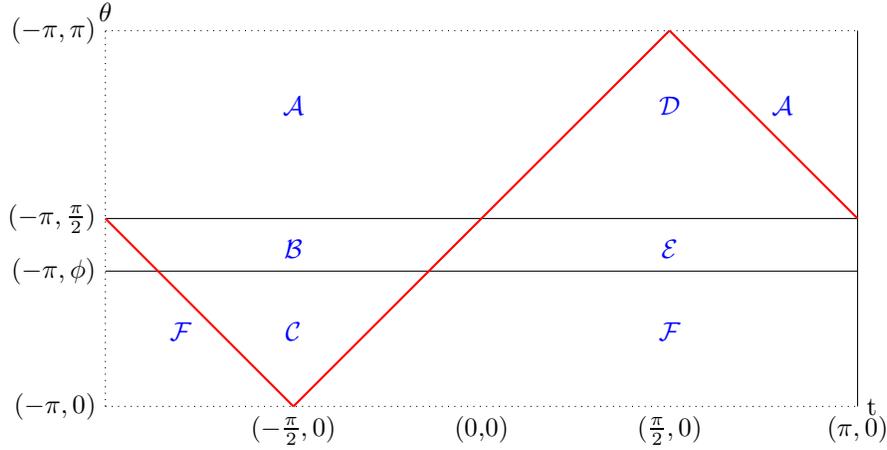
\begin{figure}[H]
\centering
\label{domain}
\begin{tikzpicture}

\draw[dotted] (-5,0)node[left]{$(-\pi, 0)$}--(0,0)node[below]{(0,0)}--(5,0)node[below]{$(\pi,0)$}node[right]{t};
\draw[dotted](-5,0)--(-5,5)node[left]{$(-\pi,\pi)$}node[above]{$\theta$};
\draw[dotted](-5,5)--(5,5);
\draw(-5,1.8)--(5,1.8);
\draw(-5,2.5)--(5,2.5);
\draw (5,0)--(5,5);
\draw[color=red][thick] (-5,2.5)--(-2.5,0);
\draw[color=red][thick] (-2.5,0)--(2.5,5);
\draw[color=red][thick] (2.5,5)--(5,2.5);
\draw[color=blue](-2.5,2.1)node{$\mathcal{B}$};
\draw[color=blue](2.5,4)node{$\mathcal{D}$};
\draw[color=blue](-2.5,4)node{$\mathcal{A}$};
\draw[color=blue](-4,1)node{$\mathcal{F}$};
\draw[color=blue](-2.5,1)node{$\mathcal{C}$};
\draw[color=blue](2.5,2.1)node{$\mathcal{E}$};
\draw[color=blue](4,4)node{$\mathcal{A}$};
\draw(-5,2.5)node[left]{$(-\pi,\frac{\pi}{2})$};
\draw (-2.5,0)node[below]{$(-\frac{\pi}{2},0)$};
\draw (2.5,0)node[below]{$(\frac{\pi}{2},0)$};
\draw[color=blue](2.5,1)node{$\mathcal{F}$};
\draw(-5,1.8)node[left]{$(-\pi,\phi)$};

\end{tikzpicture}
\end{figure}

The equation [\ref{iden}] allows one to transfer the equation $I_{g} = I_{\phi}$ into the following identity on the cylinder $\mathcal{K}:$ \begin{equation}\begin{gathered}
\frac{1}{2\pi} \int_{\mathcal{U}} (-\cot(\theta) + i)(g_{\phi}(u(t,\theta))  - h(t,\theta ))dtd\theta + \\ \frac{1}{2\pi} \int_{\mathcal{L}} (\cot(\theta) - i)( g_{\phi}(u(t,\theta ))-h(t,\theta))dtd\theta =0 , \end{gathered}
\end{equation} where the notation $h(t,\theta) = g(u(t,\theta))$ is being used. Note that $g_{\phi}(u(t,\theta))$ is the characteristic function of the set $\mathcal{C}\cup\mathcal{D}\cup\mathcal{E}.$ Separating the real and imaginary parts and using the facts that \begin{equation}\begin{gathered}\cot(\phi)>0,\ \ \cot(\theta)<0,\ \ \theta\in\mathcal{A}\cup\mathcal{D}\\ \ 0< \cot(\theta)<\cot(\phi), \ \ \theta\in\mathcal{B}\cup\mathcal{E} \\  \cot(\phi)<\cot(\theta), \ \ \theta\in\mathcal{C}\cup\mathcal{F}\end{gathered} \end{equation} one can deduce the following four part string of inequalities:
\begin{equation}\begin{gathered} -\int_{\mathcal{A}}h(t,\theta) dtd\theta - \int_{\mathcal{D}}(1-h(t,\theta) )dtd\theta + \int_{\mathcal{C}}(1-h(t,\theta) )dtd\theta +\int_{\mathcal{F}}h(t,\theta) dtd\theta \leq  \\
 -\int_{\mathcal{A}}h(t,\theta)\frac{\cot(\theta)}{\cot(\phi)}dtd\theta -\int_{\mathcal{D}}(1-h(t,\theta))\frac{\cot(\theta)}{\cot(\phi)}dtd\theta +\int_{\mathcal{C}}(1-h(t,\theta))\frac{\cot(\theta)}{\cot(\phi)}dtd\theta +\int_{\mathcal{F}}h(t,\theta)\frac{\cot(\theta)}{\cot(\phi)}dtd\theta = \\ \int_{\mathcal{B}}h(t,\theta)\frac{\cot(\theta)}{\cot(\phi)}dtd\theta +\int_{\mathcal{E}}(1-h(t,\theta))\frac{\cot(\theta)}{\cot(\phi)}dtd\theta\leq \\
\int_{\mathcal{B}}h(t,\theta)dtd\theta +\int_{\mathcal{E}}(1-h(t,\theta))dtd\theta
\end{gathered}
\end{equation}
(note the $2\pi$ factor has been dropped). The first and last parts in this string of inequalities would be equal since their difference is the imaginary part of $2\pi(I_{g}-I_{\phi)}$ and the middle equality is a result of the assumption that the real part of $I_{g}-I_{\phi}$ is zero. Since the ratio $\frac{\cot(\theta)}{\cot(\phi)}$ is less than $1$ on $\mathcal{B}\cup\mathcal{E}$ the last inequality can hold only if $h=0$ a.e, on $\mathcal{B}$ and $h=1$ a.e. on $\mathcal{E}.$ This also means that the first part of the above string is zero. Using the fact that each term in the second part is non-negative, one concludes that $h=0$ a.e. on $\mathcal{A}\cup\mathcal{F}$ and $h=1$ a.e. on $\mathcal{C}\cup\mathcal{D}.$ This completes the proof.
\end{proof}
\begin{remark} One can verify the identity (\ref{keyvalue2}) by computing the iterated integral of $\frac{\pm 1}{2\pi} (\cot(\theta)-i)$ with the appropriate sign choice over  $\mathcal{C}\cup\mathcal{D}\cup\mathcal{E}.$ The computation involves doing the $theta$ integral first. \end{remark}
The next proposition establishes the second goal of the section and implies the inequality \begin{equation} \vert 1-\exp C{g}(z,w)\vert\leq 1\end{equation} for all $z,w$ for $g\in\mathcal{G}_{1}.$
\begin{prop} Let $g$ be in $\mathcal{G}_{1}.$ Assume $g\neq g_{\theta}$ for all $0<\theta<\pi.$ Then the value
\[I_{g}=  -\frac{1}{\pi}\int_{\mathbb{C} }
\frac{g(u)}{\overline{u+1} (u-1)}  da(u),  \] is in the open convex region $\Omega_{1}$ bounded by the curve $\Gamma_{0}$ paraetrized by \[z(\theta) =\ln(2\sin\theta) +i(\frac{\pi}{2} -\theta), \ \  0<\theta <\pi.\]
\end{prop}
\begin{proof} The argument is by contradiction. Assume the value $I_g$ is not in $\Omega_{1}.$ For $0\leq\lambda\leq 1$ let $g_{\lambda}=(1-\lambda)g +\lambda g_{\frac{\pi}{2}}.$  Note that $g_{\lambda}$ is in $\mathcal{G}_{1}$ for all $\lambda\in [0,1].$
For some $\lambda_{0}$ the value $I_{g_{\lambda_{0}}}$ is in $\Gamma_{0}.$ However, by assumption $g_{\lambda_{0}}$ cannot equal $g_{\theta}$ for any $\theta.$ This contradicts the uniqueness of the values $I_{\theta}$ and completes the proof.
\end{proof}

\begin{cor} For all $z,w$  the inequality \[ \vert 1-\exp C{g}(z,w)\vert\leq 1\] holds. Moreover, it follows that equality holds in this inequality if and only if $g$ is the characteristic function of $\frac{1}{2}((z-w)D_{\theta} +(z+w))$ for some $\theta.$\end{cor}

\subsection{Connection with Operator Theory}
As mentioned in the introduction, there is a close connection between the function $E_{g}$ and operators with one dimensional self-commutator. In this subsection, we will briefly outline this connection. The reader is referred to \cite{Clancey} and the references in \cite{Clancey}, in particular, \cite{Putinar} for more detail. Let $T$ be a bounded irreducible operator on the Hilbert space $\mathcal H$ satisfying $T^*T-TT^* = \varphi\otimes\varphi,$ where $\varphi$ is an element of $\mathcal H.$ It develops that for every $\lambda$ in $\mathbb{C},$ there is a unique solution of the equation $T_{\lambda}^{*}x=\varphi$ orthogonal to the kernel of $T_{\lambda}^{*} = (T-\lambda)^{*},$ which will be denoted $T_{\lambda}^{*-1}\varphi.$ This solution satisfies $\left\Vert T_{\lambda}^{*-1}\varphi \right\Vert \leq1$ for all $\lambda.$ The $\mathcal H$-valued function $T_{\lambda}^{*-1}\varphi$  defined for all $\lambda \in \mathbb{C}$ is called the global-local resolvent associated with the operator $T^*.$  There is a unique $g=g_T \in \mathcal{G}_1$ called the Pincus principal function that satisfies 
\begin{equation} E(z,w) = \exp -\frac{1}{\pi}\int_{\mathbb{C} }
\frac{g(u)}{\overline{u-w} (u-z) }  da(u ) = 1 - (T_{w}^{*-1}\varphi, T_{z}^{*-1}\varphi)\ \ z,w\in\mathbb{C}.\end{equation}  Moreover, given $g\in\mathcal{G}_{1},$ there is a unique (up to unitary equivalence) operator $T$ with one-dimensional self-commutator as above with $g_T=g.$
Using the above formula, the function theoretic inequality \[ \vert 1-\exp C{g}(z,w)\vert\leq 1\]
follows directly from the Cauchy-Schwarz inequality 
for $(T_{w}^{*-1}\varphi, T_{z}^{*-1}\varphi).$  Further, by analyzing the case when $\vert(T_{w}^{*-1}\varphi, T_{z}^{*-1}\varphi)\vert=1,$ where equality holds in the Cauchy-Schwarz inequality, one can give an operator theoretic proof of the uniqueness of $I_{\theta}.$
\begin{remark}There are other function theoretic results about $E_g(z,w)$ that follow from connections with operator theory that await direct proof.  One of these, as stated in \cite{Putinar}, is the positivity of the kernel $K(z,w)=1-E_{g}(z,w).$ The inequality (\ref{ineq}) established here  is related to the positivity of $K(z,w)$ in the $2\times 2$ case. It would appear further function theoretic studies of this positivity will necessitate deeper investigations of the complex Cauchy transforms in (\ref{reducedvalues}).\end{remark}

Department of Mathematics

University of Georgia

Athens, GA 

email: kevinfclancey@gmail.com

\end{document}